\documentclass[11pt]{article}
\usepackage[margin=1in]{geometry}
\usepackage{amsmath,amssymb,amsthm,amsfonts}
\usepackage{cite}
\usepackage{hyperref}
\hypersetup{hidelinks}

\newtheorem{theorem}{Theorem}[section]
\newtheorem{lemma}[theorem]{Lemma}
\newtheorem{corollary}[theorem]{Corollary}
\newtheorem{proposition}[theorem]{Proposition}
\newtheorem{remark}[theorem]{Remark}
\newtheorem{example}[theorem]{Example}
\newtheorem{definition}[theorem]{Definition}

\title{Riccati Stability Without Auxiliary Matrices}
\author{Jeffrey Humpherys\\Department of Mathematics and Statistics\\Missouri University of Science and Technology\\Rolla, Missouri}
\date{\today}

\begin{document}
\maketitle

\begin{abstract}
In the 2004 collection \emph{Unsolved Problems in Mathematical Systems and Control Theory}, Erik Verriest posed the problem of characterizing Riccati stability ``without invoking additional matrices.''  We give such a characterization through a scalar invariant of the resolvent family.  The invariant is defined by covariance balances using at most $n^2+1$ frequency--direction pairs.  A finite-dimensional separation argument shows that this covariance radius equals the optimal common ellipsoidal norm of the resolvent family.  Combined with the strict bounded real lemma, this yields a necessary and sufficient condition for Riccati stability involving only the Hurwitz property of the first matrix and a single scalar inequality.  The invariant reduces to the ordinary spectral radius for one matrix, lies between the pointwise spectral-radius and unscaled small-gain levels of the resolvent family, and coincides with the operator-space spectral radius of Shalit and Shamovich for the natural resolvent function space.  The main results have been formally verified in the Lean proof assistant.
\end{abstract}



\section{Introduction}

Time delays can enter control systems in various ways, e.g., transport, measurement, computation, and communication.  They are a classical source of instability.  A feedback loop that is stable in the absence of delay can be unstable.  Often there is a critical lag that causes the system to destabilize, and so stability criteria for delayed systems are often divided into two families.  Delay-dependent criteria establish stability for all delays below a given critical bound, while delay-independent criteria establish stability for every delay, at the price of conservatism when the delay is known to be small.  The delay-independent notion is the appropriate one when delays are unknown or time-varying.  This is the setting of this paper.

Consider the linear delay system
\begin{equation}
\label{eq:delay}
\dot x(t)=Ax(t)+Bx(t-\tau),
\end{equation}
with real $A,B\in\mathbb R^{n\times n}$ and delay $\tau\ge0$, together with the quadratic Lyapunov--Krasovskii functional
\begin{equation}
\label{eq:lk}
V(x_t)=x(t)^TPx(t)+\int_{t-\tau}^{t}x(s)^TQx(s)\,ds,
\qquad P,Q>0.
\end{equation}
Differentiating along solutions of \eqref{eq:delay} gives
\[
\frac{d}{dt}V
= x(t)^T\bigl(A^TP+PA+Q\bigr)x(t) +2\,x(t)^TPB\,x(t-\tau) -x(t-\tau)^TQ\,x(t-\tau),
\]
a quadratic form in the pair $(x(t),x(t-\tau))$.  The form is negative definite for all values of that pair exactly when
\begin{equation}
\label{eq:lmi}
\begin{bmatrix}
A^TP+PA+Q & PB\\
B^TP & -Q
\end{bmatrix}<0.
\end{equation}
Because negativity is required for all values of the pair, and not merely for those produced by a solution, the criterion makes no reference to $\tau$, and one pair $(P,Q)$ yields asymptotic stability of \eqref{eq:delay} for every delay $\tau \geq 0$.  The Schur complement in the second diagonal block converts \eqref{eq:lmi} to
\begin{equation}\label{eq:ric}
A^TP+PA+Q+PBQ^{-1}B^TP<0.
\end{equation}
Verriest \cite{Verriest2004} called the pair $(A,B)$ \emph{Riccati-stable} when real symmetric $P,Q>0$ satisfying \eqref{eq:ric} exist.

When $B=0$ the integral term in \eqref{eq:lk} is unnecessary; the functional reduces to $x^T Px$ and \eqref{eq:ric} reduces to the Lyapunov inequality $A^TP+PA<0$.  Lyapunov's theorem eliminates the auxiliary matrix from that condition.  The inequality is solvable exactly when every eigenvalue of $A$ lies in the open left half-plane, a statement that refers to $A$ alone, locates the boundary of the stable region, and replaces an existential condition over a cone of matrices with a spectral computation.  In the volume \emph{Unsolved Problems in Mathematical Systems and Control Theory} \cite{Verriest2004}, Verriest posed the corresponding question for the pair $(A,B)$, requesting a characterization of Riccati stability ``without invoking additional matrices'', as the Lyapunov theorem does for the $B=0$ case.  %The obstacle is not decidability, since for any given pair \eqref{eq:lmi} is a semidefinite feasibility problem.  The obstacle is that feasibility quantifies over the unknowns $P$ and $Q$ and produces no intrinsic invariant of $(A,B)$, no analogue of the spectrum, and no visible boundary for the set of stable pairs.

The frequency domain identifies the object on which such an invariant should be built.  Assume $A$ is Hurwitz and regard \eqref{eq:delay} as the delay-free plant $\dot x=Ax+Bu$ in feedback with the pure delay $u(t)=x(t-\tau)$.  At frequency $\omega$ the delay contributes the factor $e^{-i\omega\tau}$, which has unit modulus for every $\tau$, so each loop gain that can occur is a unimodular multiple of a member of the resolvent family
\begin{equation}
\label{eq:G}
G_{A,B}(\omega)=(i\omega I-A)^{-1}B,
\qquad \omega\in\mathbb R.
\end{equation}
Norms are unchanged by unimodular scalar factors, so this family carries the gain information of the loop uniformly in the delay, and the following lemma reduces the Riccati inequality to a scaled small-gain condition on it.

\begin{lemma}
\label{lem:brl}
The pair $(A,B)$ is Riccati-stable if and only if $A$ is Hurwitz and there exists a real symmetric $Q>0$ such that
\begin{equation}\label{eq:metric}
\sup_{\omega\in\mathbb R}
\left\|Q^{1/2}G_{A,B}(\omega)Q^{-1/2}\right\|_2<1.
\end{equation}
In other words, Riccati stability is equivalent to the existence of a single quadratic metric under which the entire family $G_{A,B}(\omega)$ is contractive.  
\end{lemma}

\begin{proof}
If \eqref{eq:ric} holds, then $A^TP+PA<0$, hence $A$ is Hurwitz.  For fixed $Q>0$, \eqref{eq:ric} is the strict bounded real inequality for the realization
\[
\bigl(A,BQ^{-1/2},Q^{1/2},0\bigr),
\]
whose transfer matrix is $Q^{1/2}(sI-A)^{-1}BQ^{-1/2}$.  The strict bounded real lemma \cite[Cor.~13.24]{ZhouDoyleGlover1996} therefore gives \eqref{eq:metric}, and conversely.
\end{proof}

This paper is organized as follows:  In Section~2 we construct, for an arbitrary compact family of matrices, a scalar invariant defined by finite balances of transformed covariances and prove that it equals the optimal common ellipsoidal norm.  In Section~3, we apply the construction to the resolvent family and characterize Riccati stability by the Hurwitz property of $A$ and one scalar inequality, with failure exhibited by finitely many frequencies and directions.  In Section~4, we evaluate the invariant exactly in a two-dimensional example, and finally in Section~5 we discuss computation and an independent operator-space description of the same quantity.  All results in Sections 2--4, including the exact identities and inequalities in the example, have been formally verified in the Lean~4 proof assistant using the Mathlib library~\cite{RiccatiLean}.

\section{The covariance radius}

From this point onward, complex matrices and vectors enter explicitly.  For a complex matrix $X$, we write $X^H$ for its Hermitian (conjugate) transpose; when $X$ is real, $X^H=X^T$.

The elimination achieved by Lyapunov's theorem can be read through rank-one identities.  If $Xz=\lambda z$ with $|\lambda|=r(X)$, then
\[
Xzz^HX^H=|\lambda|^2zz^H=r(X)^2\,zz^H,
\]
so a single direction already exhibits amplification of its covariance by the factor $r(X)^2$, with no metric involved.  The classical identity $\inf_{H>0}\|H^{1/2}XH^{-1/2}\|_2=r(X)$ states that no metric can improve on this factor, so for one matrix the best covariance amplification and the best ellipsoidal norm coincide.  For a family of matrices a single direction no longer suffices, because the direction that one member amplifies may be handled cheaply by a metric adapted to the others.  The correct generalization permits finitely many members of the family to act on finitely many directions and compares aggregate covariances.  
Throughout this section, let $n\ge1$ and let $\Sigma\subset\mathbb C^{n\times n}$ be compact and contain zero.

\begin{definition}
\label{def:cov}
Define
\begin{equation}\label{eq:covdef}
\rho_{\rm cov}(\Sigma)^2
=
\sup\left\{
 t\ge0:
 \sum_{j=1}^m X_jz_jz_j^HX_j^H
 \succeq
 t\sum_{j=1}^m z_jz_j^H
\right\},
\end{equation}
where the supremum is taken jointly over
\[
1\le m\le n^2+1,
\qquad X_j\in\Sigma,
\qquad z_j\in\mathbb C^n,
\]
with the $z_j$ not all zero.
\end{definition}

The definition contains no common metric or similarity transformation.  It records how much aggregate covariance can be forced by finitely many matrix--direction pairs.  The restriction $m\le n^2+1$ reflects the dimension $n^2$ of the real vector space of Hermitian matrices, and Theorem~\ref{thm:covdual} shows that larger supports are never needed.  Taking traces shows that every feasible $t$ satisfies $t\le\max_{X\in\Sigma}\|X\|_2^2$, so the supremum is finite.  By homogeneity we may normalize $\sum_j\|z_j\|_2^2=1$ and pad with zero vectors to $m=n^2+1$; compactness then shows that the supremum is attained.

\begin{theorem}
\label{thm:covdual}
For every compact $\Sigma\subset\mathbb C^{n\times n}$ containing zero,
\begin{equation}
\label{eq:covdual}
\rho_{\rm cov}(\Sigma) = \inf_{H>0}\ \sup_{X\in\Sigma} \left\|H^{1/2}XH^{-1/2}\right\|_2,
\end{equation}
where the infimum is over Hermitian positive definite matrices $H$.
\end{theorem}

\begin{proof}
Let the right side of \eqref{eq:covdual} be $\alpha$.  If
\[
\sum_jX_jz_jz_j^HX_j^H\succeq t\sum_jz_jz_j^H,
\]
then for any $H>0$, with
\[
\gamma_H=\sup_{X\in\Sigma}\|H^{1/2}XH^{-1/2}\|_2,
\]
we have $X^HHX\preceq\gamma_H^2H$.  Taking the trace of the covariance inequality against $H$ gives $t\le\gamma_H^2$.  Hence $\rho_{\rm cov}(\Sigma)\le\alpha$.

Conversely, fix $0<t<\alpha^2$ and consider
\[
\mathcal C_t=
\{tzz^H-Xzz^HX^H:X\in\Sigma,\ \|z\|_2=1\},
\]
a compact set whose convex hull is compact as well.  If $0\notin\operatorname{conv}\mathcal C_t$, strict separation gives a Hermitian $H$ and $\varepsilon>0$ such that
\[
\operatorname{tr}\!\left[H(tzz^H-Xzz^HX^H)\right]\ge\varepsilon
\]
for every $X\in\Sigma$ and unit $z$.  Taking $X=0$ shows $H>0$, while the same inequality gives
\[
X^HHX\preceq tH-\varepsilon I<tH
\qquad(X\in\Sigma),
\]
so that $\sup_{X\in\Sigma}\|H^{1/2}XH^{-1/2}\|_2^2<t$, contradicting $t<\alpha^2$.  Thus $0\in\operatorname{conv}\mathcal C_t$.

The real vector space of Hermitian $n\times n$ matrices has dimension $n^2$.  Carath\'eodory's theorem therefore gives $m\le n^2+1$, matrices $X_j\in\Sigma$, unit vectors $u_j$, and $a_j\ge0$ with $\sum_j a_j=1$ such that
\[
\sum_j a_j(tu_ju_j^H-X_ju_ju_j^HX_j^H)=0.
\]
Setting $z_j=\sqrt{a_j}\,u_j$ shows that $t$ is feasible in \eqref{eq:covdef}.  Since every $t<\alpha^2$ is feasible, $\rho_{\rm cov}(\Sigma)\ge\alpha$.
\end{proof}

The proof exhibits metrics and covariance balances as dual objects.  A metric under which the family contracts by the factor $\sqrt t$ separates the origin from $\mathcal C_t$, and when no such metric exists the origin lies in the convex hull of $\mathcal C_t$, which is a finite balance of transformed covariances.  The optimal ellipsoidal norm is therefore not an upper bound produced by one fortunate choice of $H$ but the exact value of the best finite balance.

\begin{remark}\label{rem:single}
For a single matrix $X$, Theorem~\ref{thm:covdual} reduces to the classical identity
\[
\rho_{\rm cov}(\{0,X\})
=
\inf_{H>0}\|H^{1/2}XH^{-1/2}\|_2
=r(X).
\]
Thus the covariance radius extends the ordinary spectral radius from one matrix to the problem of finding one quadratic metric for an entire family.
\end{remark}

A chain of directions extends the eigenvector construction from one matrix to products, and shows that finite balances already reach every spectral quantity of the family.

\begin{lemma}\label{lem:word}
For any $X_1,\ldots,X_k\in\Sigma$,
\[
\rho_{\rm cov}(\Sigma)\ \ge\ r(X_1\cdots X_k)^{1/k}.
\]
\end{lemma}

\begin{proof}
Set $W=X_1\cdots X_k$ and assume $r(W)>0$, the bound being trivial otherwise.  Choose $Wv=\mu v$ with $|\mu|=r(W)$ and $v\ne0$, put $v_k=v$, and define $v_{j-1}=X_jv_j$ for $j=k,\ldots,1$, so that $v_0=Wv=\mu v_k$ and no $v_j$ vanishes.  With $t=|\mu|^{2/k}$ and $z_j=t^{(j-1)/2}v_j$,
\[
\sum_{j=1}^kX_jz_jz_j^HX_j^H
=\sum_{j=1}^kt^{j-1}v_{j-1}v_{j-1}^H
=|\mu|^2v_kv_k^H+\sum_{j=1}^{k-1}t^{j}v_{j}v_{j}^H
=t\sum_{j=1}^kt^{j-1}v_jv_j^H
=t\sum_{j=1}^kz_jz_j^H,
\]
using $|\mu|^2=t^k$ in the third equality.  Pairing this balance with an arbitrary $H\succ0$ as in the first half of the proof of Theorem~\ref{thm:covdual} gives $t\le\sup_{X\in\Sigma}\|H^{1/2}XH^{-1/2}\|_2^2$, an argument that places no restriction on the number of terms, and taking the infimum over $H$ gives $t\le\rho_{\rm cov}(\Sigma)^2$.
\end{proof}

At $k=1$ the lemma gives $\sup_{X\in\Sigma}r(X)\le\rho_{\rm cov}(\Sigma)$.  Dual positive-semidefinite constructions of this type are classical in common quadratic Lyapunov theory; see \cite{KamenetskiyPyatnitskiy1987,MoldovanGowda2010}.  The present contribution is the scalar formulation \eqref{eq:covdef} and its finite support bound.

\section{Main characterization and bounds}

For Hurwitz $A$, let
\[
\Sigma_{A,B}
=
\{G_{A,B}(\omega):\omega\in\mathbb R\}\cup\{0\}.
\]
Because $G_{A,B}(\omega)\to0$ as $|\omega|\to\infty$, this set is compact.  Define
\begin{equation}\label{eq:rhoR}
\rho_R(A,B)=\rho_{\rm cov}(\Sigma_{A,B}).
\end{equation}

\begin{theorem}
\label{thm:main}
For real $A,B\in\mathbb R^{n\times n}$,
\begin{equation}\label{eq:main}
(A,B)\text{ is Riccati-stable}
\quad\Longleftrightarrow\quad
A\text{ is Hurwitz and }\rho_R(A,B)<1.
\end{equation}
\end{theorem}

\begin{proof}
By Theorem~\ref{thm:covdual},
\[
\rho_R(A,B)
=
\inf_{H>0}\sup_\omega
\|H^{1/2}G_{A,B}(\omega)H^{-1/2}\|_2.
\]
For real $A,B$, $G_{A,B}(-\omega)=\overline{G_{A,B}(\omega)}$.  Hence if a Hermitian $H>0$ gives a strict bound $\gamma$, then so does $\overline H$ (evaluate the bound at $-\omega$ and take complex conjugates), and their average
\[
Q=\tfrac12(H+\overline H)
\]
is real symmetric positive definite and gives the same strict bound.  Thus the infimum may be taken over real symmetric $Q>0$.  Lemma~\ref{lem:brl} now gives \eqref{eq:main}.
\end{proof}

Theorem~\ref{thm:main} is the requested analogue of the Lyapunov theorem.  The Hurwitz condition disposes of the delay-free dynamics, the scalar inequality $\rho_R(A,B)<1$ replaces the existential condition over pairs $(P,Q)$, and both conditions refer only to quantities generated by $A$ and $B$.

\begin{corollary}
\label{cor:finitecert}
For real $A,B\in\mathbb R^{n\times n}$, the pair $(A,B)$ is not Riccati-stable if and only if either $A$ is not Hurwitz or there exist $m\le n^2+1$, frequencies $\omega_1,\ldots,\omega_m\in\mathbb R$, and vectors $z_1,\ldots,z_m\in\mathbb C^n$, not all zero, such that
\[
\sum_{j=1}^m G_{A,B}(\omega_j)z_jz_j^HG_{A,B}(\omega_j)^H
\succeq
\sum_{j=1}^m z_jz_j^H.
\]
\end{corollary}

\begin{proof}
If $A$ is Hurwitz and $(A,B)$ is not Riccati-stable, Theorem~\ref{thm:main} gives $\rho_R(A,B)\ge1$.  Since the supremum in Definition~\ref{def:cov} is attained, there exist maximizing $X_j$ and $z_j$ with $t=\rho_R(A,B)^2\ge1$, and these satisfy the displayed inequality.  Any terms with $X_j=0$ may be discarded, leaving $X_j=G_{A,B}(\omega_j)$ at finite real frequencies.  The converse follows immediately from Definition~\ref{def:cov} and Theorem~\ref{thm:main}.
\end{proof}

Corollary~\ref{cor:finitecert} is the counterpart of exhibiting an eigenvalue in the closed right half-plane.  Failure of Riccati stability always reduces to a single semidefinite inequality in at most $n^2+1$ frequencies and directions, so both directions of the characterization are expressed in data generated by the pair.

The same variational formula places $\rho_R$ between two familiar frequency-domain quantities.  Lemma~\ref{lem:word} with $k=1$ gives the pointwise lower bound, while the identity metric gives the unscaled upper bound.

\begin{proposition}
\label{prop:bounds}
For Hurwitz $A$,
\begin{equation}\label{eq:sandwich}
\sup_{\omega\in\mathbb R}r(G_{A,B}(\omega))
\le
\rho_R(A,B)
\le
\sup_{\omega\in\mathbb R}\|G_{A,B}(\omega)\|_2.
\end{equation}
More generally, for every finite choice of frequencies $\omega_1,\ldots,\omega_k$,
\begin{equation}\label{eq:resolventword}
r\!\left(G_{A,B}(\omega_1)\cdots G_{A,B}(\omega_k)\right)^{1/k}
\le\rho_R(A,B).
\end{equation}
\end{proposition}

\begin{proof}
The finite-product estimate is Lemma~\ref{lem:word} applied to the resolvent family $\Sigma_{A,B}$.  Its case $k=1$ gives the left inequality in \eqref{eq:sandwich}.  The right inequality follows from Theorem~\ref{thm:covdual} by choosing the identity metric $H=I$.
\end{proof}

The leftmost quantity in \eqref{eq:sandwich} governs strong delay-independent stability \cite{ChenLatchman1995}, while the rightmost quantity is the unscaled small-gain level.  Consequently
\[
\text{unscaled small gain}
\Longrightarrow
\text{Riccati stability}
\Longrightarrow
\text{strong delay-independent stability}.
\]
The example in the next section separates these quantities numerically.

\begin{remark}
Theorem~\ref{thm:main} recovers the partial results in \cite{Verriest2004}
as corollaries.  Since $G_{A,B}(0)=-A^{-1}B$, the lower bound in
\eqref{eq:sandwich} gives $r(A^{-1}B)\le\rho_R(A,B)$, so Riccati stability
forces $A^{-1}B$ to be Schur stable.  Moreover
$G_{\alpha A,\alpha B}(\omega)=G_{A,B}(\omega/\alpha)$ for $\alpha>0$, so
the resolvent family and hence $\rho_R$ are invariant under joint scaling,
and similarity invariance of $\rho_R$ follows from the variational formula
\eqref{eq:covdual}.  These sharpen the scaling and similarity lemmas of
\cite{Verriest2004} from preservation statements to exact invariance of
the characterizing scalar.
\end{remark}

\section{A two-dimensional example}

The example below serves three purposes.  The outer bounds in \eqref{eq:sandwich} are strict for this pair, the covariance radius is evaluated exactly through a particular metric, and a scaled version of the pair is strongly delay-independent stable without being Riccati-stable.

\begin{example}\label{ex:2d}
Let
\[
A=\begin{bmatrix}-3&3\\-3&-1\end{bmatrix},
\qquad
B=\begin{bmatrix}1&-\tfrac12\\-\tfrac12&-\tfrac12\end{bmatrix}.
\]
The eigenvalues of $A$ are $-2\pm2\sqrt2\,i$.  Numerically,
\[
\sup_\omega r(G_{A,B}(\omega))\approx0.272,
\qquad
\rho_R(A,B)\approx0.453,
\qquad
\sup_\omega\|G_{A,B}(\omega)\|_2\approx0.502.
\]
Thus $\sup_\omega r(G_{A,B}(\omega))<\rho_R(A,B)<\sup_\omega\|G_{A,B}(\omega)\|_2$.

For
\[
Q=\begin{bmatrix}28&-5\\-5&13\end{bmatrix},
\]
a direct calculation gives, for every real $\omega$,
\begin{equation}\label{eq:example_exact}
\left\|Q^{1/2}G_{A,B}(\omega)Q^{-1/2}\right\|_2^2
=
r\!\left(G_{A,B}(-\omega)G_{A,B}(\omega)\right).
\end{equation}
The left side gives an upper bound, since any single metric bounds the infimum from above, while Lemma~\ref{lem:word} at $k=2$, applied to the members $G_{A,B}(-\omega)$ and $G_{A,B}(\omega)$, gives the matching lower bound $r(G_{A,B}(-\omega)G_{A,B}(\omega))^{1/2}\le\rho_R(A,B)$ at every frequency.  Hence
\[
\rho_R(A,B)
=
\sup_{\omega\in\mathbb R}
r\!\left(G_{A,B}(-\omega)G_{A,B}(\omega)\right)^{1/2}
\approx0.45278,
\]
with the maximizing frequency near $|\omega|=2.8754$.  
By homogeneity, $\rho_R(A,cB)=|c|\rho_R(A,B)$.  Taking $c=5/2$ gives
\[
\sup_\omega r(G_{A,cB}(\omega))\approx0.679<1
<1.13\approx\rho_R(A,cB).
\]
Thus $(A,\tfrac52B)$ is strongly delay-independent stable but is not Riccati-stable.  The covariance radius therefore measures a genuine gap between pointwise spectral stability and the existence of one common quadratic metric.
\end{example}

\section{Discussion}

Theorem~\ref{thm:main} resolves Verriest's elimination problem with a scalar invariant whose definition contains neither of the positive-definite matrices in \eqref{eq:ric}.  Theorem~\ref{thm:covdual} shows that this intrinsic covariance quantity is exactly the optimal common ellipsoidal norm, while Carath\'eodory's theorem shows that every obstruction requires at most $n^2+1$ rank-one frequency--direction terms.  Numerically, $\rho_R$ can still be computed by bisection on $\gamma$, since
\[
\rho_R(A,B)<\gamma
\quad\Longleftrightarrow\quad
(A,B/\gamma)\text{ is Riccati-stable},
\]
with the LMI \eqref{eq:lmi} solved at each step.

There is also an operator-space interpretation.  Let $\mathcal E_{A,B}\subset C(\mathbb R\cup\{\infty\})$ be the span of the scalar entries of $G_{A,B}$, with its inherited matrix norms, and regard $G_{A,B}$ as an element of $M_n(\mathcal E_{A,B})$ through its entries.  Shalit and Shamovich \cite{ShalitShamovich2025} define an operator-space spectral radius whose similarity theorem gives, after scaling,
\[
\rho_{\mathcal E}(X)
=
\inf_{S\in GL_n(\mathbb C)}\|S^{-1}XS\|_{M_n(\mathcal E)}.
\]
For $X=G_{A,B}$, the matrix norm is $\sup_\omega\|\cdot\|_2$.  Taking $H=S^{-H}S^{-1}$ identifies the similarity norm with the corresponding quadratic-metric norm, so Theorem~\ref{thm:covdual} gives
\[
\rho_{\mathcal E_{A,B}}(G_{A,B})=\rho_R(A,B).
\]
Thus the recent operator-space radius and the finite covariance radius give two descriptions of the same scalar; the latter is sufficient for the proof above.  The two constructions arose independently and for different purposes, which indicates that the scalar is a natural invariant of the family rather than an artifact of the delay problem.

The covariance viewpoint also distinguishes Riccati stability from ordinary spectral tests.  The pointwise spectral radius involves one frequency at a time.  The finite-word bound in Lemma~\ref{lem:word} follows one chain of directions through finitely many resolvent values.  A covariance balance is more flexible: several frequency--direction pairs may act simultaneously.  This is precisely the additional obstruction created by requiring one common quadratic metric.  Natural questions are whether the general support bound $n^2+1$ can be reduced for resolvent families and when a finite-word lower bound is already exact.

\section*{Acknowledgments and disclosure of AI assistance}

The generative AI assistants Google Gemini, OpenAI ChatGPT, and Anthropic Claude were used in the development, drafting, and checking of this manuscript, including literature searches and numerical verification. Anthropic Claude was also used to formalize the results of the paper in the Lean~4 proof assistant with the Mathlib library, and the formal proofs were checked by Lean. These tools were not relied upon as authorities for correctness, and responsibility for all mathematical claims rests with the author.

\end{document}